\documentclass[12pt,reqno]{amsart}
\usepackage{hyperref}
\usepackage{amsmath,amssymb, amsthm}
\usepackage{caption}
\usepackage{amssymb,amsmath,euscript,enumerate,tikz}
\usepackage[margin=1in]{geometry}
\usepackage{graphicx}
\usepackage{float}
\usepackage{pgf,tikz}
\usepackage{mathrsfs}
\usepackage{mathtools}

\newcommand \ass{\operatorname{Ass}}

\newcommand \R{\mathbb{R}}

\newcommand \K{\mathbb{K}}

\newcommand{\y}{\mathbf{y}}

\DeclareMathOperator{\sfBorel}{sfBorel}
\DeclareMathOperator{\mon}{Mon}
\DeclareMathOperator{\lex}{lex}
\DeclareMathOperator{\sflex}{sfLex}

\theoremstyle{plain}
\newtheorem{theorem}{Theorem}[section]
\newtheorem{lemma}[theorem]{Lemma}

\newtheorem{corollary}[theorem]{Corollary}

\theoremstyle{definition}

\newtheorem{remark}[theorem]{Remark}
\newtheorem{construction}[theorem]{Construction}
\newtheorem{definition}[theorem]{Definition}
\newtheorem{example}[theorem]{Example}

\begin{document}

\title[The Waldschmidt Constant of Square-free Principal Borel Ideals]{The Waldschmidt Constant of Square-free Principal Borel Ideals}

 \email{ajay.kumar@iitjammu.ac.in}
 \email{rajiv.kumar@iitjammu.ac.in}
 \email{2022rma2004@iitjammu.ac.in, paramjam2018@gmail.com}
 \author[Ajay Kumar]{Ajay Kumar}
  \author[Rajiv Kumar]{Rajiv Kumar}
 \author[Paramhans Kushwaha]{Paramhans Kushwaha}
\address{Indian Institute of Technology, Jammu}
\address{NH-44, PO Nagrota, Jagti, Jammu and Kashmir 181221}
\curraddr{}
\email{}
\thanks{}


\date{}

\dedicatory{}
\begin{abstract}
	Fix a square-free monomial $m = x_{i_1} \cdots x_{i_s}$ in $S = \K[x_1, \dots, x_n]$, where $\K$ is a field. The ideal generated by all square-free monomials that are Borel moves of $m$ is referred to as the \emph{square-free principal Borel ideal}, denoted by $\mathrm{sfBorel}(m)$. While the Waldschmidt constant of $\mathrm{sfBorel}(m)$ has been partially studied in the literature, we consider the remaining cases and provide bounds for the Waldschmidt constant of $\mathrm{sfBorel}(m)$. In some cases, we explicitly compute its exact value. Furthermore, we study this invariant for square-free Borel and square-free lex-segment ideals, establishing both bounds and exact values for special cases.
	\end{abstract}
	\maketitle
\section{Introduction}
  The Waldschmidt constant was introduced in $1970$'s  by Waldschmidt \cite{Waldschmidt-introduced} for the study of ideals of finite point sets in the context of complex analysis. This invariant has become an important tool in understanding the asymptotic properties of homogeneous ideals. More precisely, given a homogeneous ideal $I$ in $ S=\K[x_1,\ldots,x_n]$, denote $\alpha(I)$ the smallest degree of a generator of $I$. Then it is known that the limit $\lim_{m\rightarrow{\infty}}\frac{\alpha(I^{(m)})}{m}$ exists, where $I^{(m)}$ denotes the $m$-th symbolic power of $I$. This limit is referred to as the \emph{Waldschmidt constant} of $I$, denoted by $\widehat{\alpha}(I)$. The Waldschmidt constant provides insight into the containment problem: it can be used to compare ordinary powers with the symbolic powers of an ideal. This invariant has appeared in different areas of mathematics, e.g., in number theory (\cite{Singular-points-on-complex-hypersurfaces-and-multidimensional-Schwartz-lemma-Chudnovsky, Numbers-transcendants-..-Waldschmidt}), complex analysis \cite{Estimation-L2-pour-..-Skoda}, algebraic geometry (\cite{Comparing-powers-and-symbolic-powers-of-ideals-B_H, The-resurgence-of-ideals-of-points-and-the-containment-problem-Bocci-Harbourne, A-canonical-linear-system-associated-to-adjoint-divisors-in-characteristic-p>0-Schwede}), and commutative algebra \cite{Are-symbolic-powers-highly-evolved-Harbourne-Huneke}. For any homogeneous ideal $I$, Bocci and Harbourne established various bounds for $\widehat{\alpha}(I)$ (see \cite{Comparing-powers-and-symbolic-powers-of-ideals-B_H}), sparking renewed interest in its computation. For recent developments on the Waldschmidt constant, we refer readers to (see \cite{Steiner-systems-and-configrations-of-points-BFGM, Chudnovsky-conjecture-and-the-stable-Harbourne-Huneke-containment-BGHN, BCGHJNSTV2016-The-Waldschmidt-constant-for-sqfree-monomial-ideals, Powers-of-principal-Q-Borel-ideals-CKCSV, IdealsofPowers-and-Powersofideals-,The-Waldschmidt-Constant-of-Special-K-configrations-in-$P^n$-CGS, Borel-Generators-FMJ-2011}), and references therein. 

    The goal of this paper is to investigate the Waldschmidt constant of the square-free Borel and square-free lex-segment ideals (see Definition \ref{def: Borel ideals} and \ref{def: square-free lex ideal}). Given a monomial $m$ in $S$, a \emph{Borel move} of $m$ is a monomial $m\cdot\frac{x_{i_1}}{x_{j_1}}\cdots \frac{x_{i_s}}{x_{j_s}}$, where $i_t<j_t$, and $x_{j_t}$ divides $m$ for all $t$. 
    A monomial ideal $I$ is a (\emph{square-free}) \emph{principal Borel ideal} if there exists a (\emph{square-free}) monomial $m$ such that $I$ is generated by all the (\emph{square-free}) Borel moves of $m$. 
    For more on the topic, we point the readers to (\cite{BCGHJNSTV2016-The-Waldschmidt-constant-for-sqfree-monomial-ideals,  Waldschmidt-constants-for-Stanley-Reisner-ideals-of-a-class-of-simplicial-complexes, CKSV2022-On-the-Waldschmidt-constant-of-sqfree-principal-Borel-ideals, Symbolic-powers-of-monomial-ideals-Cooper-Embree-Robert-Ha-Hoefel, Borel-Generators-FMJ-2011}) and references therein.

   In this work, we focus on square-free Borel ideals and square-free lex-segment ideals. As shown in \cite{BCGHJNSTV2016-The-Waldschmidt-constant-for-sqfree-monomial-ideals}, 
the primary decomposition of a square-free monomial ideal $I$ can be used 
to construct an optimization problem. The optimal solution to this problem provides the Waldschmidt constant of $I$. Solving this optimization problem can be computationally challenging. For example, if $m = x_2 x_{n-1} x_n$, $n\geq 5$, then finding $\widehat{\alpha}(\mathrm{sfBorel}(m))$ requires solving a linear optimization problem with $n$ variables and $2n-2$ constraints. Although the degree of $m$ is only $3$, the complexity of this problem grows significantly with $n$. Therefore, it is natural to ask what information can be obtained about the Waldschmidt constant in such cases.

Let $m=x_{i_1}\cdots x_{i_s}$ be a square-free monomial. Authors in \cite{CKSV2022-On-the-Waldschmidt-constant-of-sqfree-principal-Borel-ideals} have obtained an upper bound and a recursive method for a lower bound for $\widehat{\alpha}(\sfBorel(m))$. However, there is no concise lower bound in the literature. Suppose $t$ is the smallest number such that $i_{t}+2\leq i_{t+1}$. Then, the authors have computed the Waldschmidt constant $\widehat{\alpha}(\sfBorel(m))$ provided $i_t\geq s$. Motivated by their work, we consider the case when $i_t< s$ and provide bounds for the Waldschmidt constant. In some cases, we compute the exact values of the Waldschmidt constant $\widehat{\alpha}(\sfBorel(m))$. These bounds are expressed in terms of $s=\deg(m)$ and $i_j,1\leq j\leq s$. The main tool in the proofs is the associated linear optimization problem and the description of associated primes of $\sfBorel(m)$ given in \cite{BCGHJNSTV2016-The-Waldschmidt-constant-for-sqfree-monomial-ideals}. In fact, we prove the following:\vspace{2mm}\newline 
\textbf{Theorem \ref{thm: lower-bound in general}}
Let $m=x_{i_1}x_{i_2}\cdots x_{i_s}$ be a square-free monomial with $T'(m)$, and $l$ be largest such that $i_{t_{k_1}}+1\leq t_l$, and suppose $I=\sfBorel(m)$. Then  
        $$\widehat{\alpha}(I)\geq 1+\frac{s-{{i_{t_{k_1}}-1}}}{i_{t_l}-t_l+1}+\frac{i_{t_k}}{i_{t_{k}}-t_k+1}+\sum_{j=1}^{r-1}\frac{i_{t_{k_j}}-i_{t_{k_{j+1}}}}{i_{t_{k_j}}-{t_{k_{j}}}+1}.$$  
        \vspace{2mm}\newline 
 \textbf{Theorem \ref{thm: Upper bound on Waldschmidt constant with conditions}}       
 Let $m=x_{i_1}x_{i_2}\cdots x_{i_s}$ be a square-free monomial with $2\leq i_1\leq s-1$,
   and $I=\sfBorel(m)$. Suppose $i_j-i_1\geq (j-1)(i_2-i_1)$, for all $3\leq j\leq i_1+1$. Then $$\widehat{\alpha}(I)\leq 2+\frac{s-i_1-1}{i_1(i_2-i_1)}.$$ In addition, if $i_{i_1+1}\geq s-1$ and $i_1(i_2-i_1)=i_{i_1+1}-i_1$, then the equality holds.\vspace{3mm}
   
  Example \ref{exm: Lower bound is attained for infinitely} presents an infinite family of square-free Borel ideals for which the bound in Theorem \ref{thm: lower-bound in general} is attained, while the bound given in \cite[Theorem 4.4]{CKSV2022-On-the-Waldschmidt-constant-of-sqfree-principal-Borel-ideals} is strict. In fact, the following corollary produces infinitely many classes of square-free Borel ideals for which the lower bound in Theorem \ref{thm: lower-bound in general} is achieved.
\vspace{1mm}\newline 
\textbf{Corollary \ref{cor: Waldschmidt constant}}
Let $m=x_{i_1}\cdots x_{i_s}$ with $T(m)$, $IT(m)$, $T'(m)$, and $l$ be the largest such that $i_{t_{k_1}+1}\leq t_l$, and let $I=\sfBorel(m)$. If $T'(m)$ is a truncation of $T(m)$ and there is no $t_p\in T(m)$ such that ${t_{k_1}}<t_p<t_l$, then  $$\widehat{\alpha}(I)= 1+\frac{s-i_{t_{k_1}}-1}{i_{t_l}-t_l+1}+\frac{i_{t_k}}{i_{t_{k}}-t_k+1}+\sum_{j=1}^{r-1}\frac{i_{t_{k_j}}-i_{t_{k_{j+1}}}}{i_{t_{k_j}}-{t_{k_{j}}}+1}.$$ 
Further, we investigate the Waldschmidt constant of square-free Borel ideals and square-free lex-segment ideals. We provide upper and lower bounds for the Waldschmidt constant of these ideals. These bounds are in terms of the Waldschmidt constant of some principal square-free Borel ideals. In some cases, we obtain the exact values. In fact:\vspace{3mm}\newline 
\textbf{Theorem \ref{thm: the Waldschmidt constant of lex ideals}}
 Let $m=x_{i_1}\cdots x_{i_s}$ be a square-free monomial in $S$. Then 
     \begin{enumerate}
         \item $\widehat{\alpha}(\sflex(m))\leq 1+\frac{s-1}{i_1}$, and 
         \item $\widehat{\alpha}(\sflex(m))\geq 
         \begin{cases}
             1+\frac{s-1}{i_1}, \qquad\text{ if } i_{1}\geq s\\
             2+\frac{s-i_1-1}{n-s+1},\quad\text{ if } i_{1}\leq s-1.
         \end{cases}$
     \end{enumerate}
     In particular, if $i_1\geq s-1$, then $\widehat{\alpha}(\sflex(m))=1+\frac{s-1}{i_1}.$
     \vspace{3mm}
        
    Our article is organised as follows: Section \ref{section: preliminaries} reviews the essential background on square-free Borel ideals and the Waldschmidt constant. In Section \ref{section: lower bound}, we establish a lower bound for the Waldschmidt constant. Section \ref{section: upper bound} is devoted to studying an upper bound for the Waldschmidt constant. In Section \ref{section: the waldschmidt constant}, exact values of 
$\widehat{\alpha}(\text{sfBorel}(m))$ are determined under specific conditions. In Section \ref{sec: square-free Borel and lex ideals}, bounds on the Waldschmidt constant for square-free Borel and square-free lex-segment ideals with equality under certain conditions are obtained.
   \section{Preliminaries}\label{section: preliminaries}
    Throughout this article, let $S = \K[x_1, \dots, x_n]$ be a polynomial ring over a field $\K$. This section reviews the essential notation for square-free monomial ideals, square-free Borel ideals, and the Waldschmidt constant. We specifically emphasize the definition of square-free Borel ideals, as they constitute the primary focus of this study.

    \begin{definition}\label{def: Borel ideals}
        Let $T=\{m_1,\dots,m_r\}$ be a set of square-free monomials in $S$. Define $\sfBorel(T):=\sfBorel(m_1,\ldots, m_r)$, the \emph{square-free Borel ideal generated by $T$}, to be the ideal generated by all square-free monomials that can be obtained from the Borel moves of elements of $T$. 
        If $T=\{m\}$, then $\sfBorel(T)$ is the \emph{square-free principal Borel ideal} generated by $m$.
    \end{definition}

    We recall two tuples from \cite[Definition 2.2]{CKSV2022-On-the-Waldschmidt-constant-of-sqfree-principal-Borel-ideals} that are essential in this article.
    \begin{definition}
       Let $m=x_{i_1}\cdots x_{i_s}$ be a square-free monomial in $S$. The tuple $T(m)=(t_0,\ldots,t_k)$ is inductively defined as $t_0=s$ and suppose $t_{i-1}$ is already defined, then $t_i=\max\{j<t_{i-1}~|~ i_j<i_{j+1}-1\}$ and $IT(m)=(i_{t_0},\ldots,i_{t_k})$.
    \end{definition}  
    
    The sequence $T(m)$ tracks the positions where the indices are increased by more than one. Note that $T(m)$ and $IT(m)$ are decreasing while the indices of $m$ are increasing. This sometimes creates confusion while tracing back elements of $IT(m)$ in $m$. Therefore, for simplicity, we use these tuples in reverse order and write $T(m)=(t_k,\ldots,t_0)$ and $IT(m)=(i_{t_k},\ldots,i_{t_0})$. Observe that both $T(m)$ and $(i_{t_k}-t_k,i_{t_{k-1}}-t_{k-1},\ldots,i_{t_0}-t_0)$ are strictly increasing sequences. For the proof, we refer to \cite[Lemma 2.4]{CKSV2022-On-the-Waldschmidt-constant-of-sqfree-principal-Borel-ideals}. 

    Given an ideal $I$ in $S$, let $\ass(I)$ denote the \emph{set of associated primes} of the ideal $I$.  When $I$ is a radical ideal, it is known that all the associated primes of $I$ are minimal. In particular, if $I$ is a square-free monomial ideal, then every minimal prime of $I$ is generated by a set of variables. The following result gives the associated primes of $\sfBorel(m)$ that are described in \cite[Theorem 3.17]{Borel-Generators-FMJ-2011} as follows.
    \begin{theorem}\label{thm: description of primes}
        Let $m=x_{i_1}\cdots x_{i_s}$ be a square-free monomial in the polynomial ring $S$ with $T(m)=(t_k,\ldots,t_0),$ $IT(m)=(i_{t_k},\ldots,i_{t_0})$, and $I=\sfBorel(m)$. Then $\langle x_{j_1},\ldots,x_{j_l}\rangle\in \ass(I)$ if and only if $x_{j_1}x_{j_2}\cdots x_{j_l}$ is a minimal generator of the square-free Borel ideal $$\sfBorel(\{x_j\cdots x_{i_j}:1\leq j\leq s\})=\sfBorel(\{x_{t_j}x_{t_j+1}\cdots x_{i_{t_j}}:0\leq j\leq k\}).$$
    \end{theorem}
    We say that a prime ideal $Q$ is \emph{associated to} a monomial $m=x_{j_1}\cdots x_{j_k}$ if $Q=\langle x_{j_1}, \ldots, x_{j_k}\rangle$. 
    Since for each $0 \leq j \leq k$, the monomials $x_{t_j}\cdots x_{i_{t_j}}$ have different degrees, they are irredundant as generators.
    Therefore, for each $0\leq j\leq k$, the prime ideal associated to  the monomial $ x_{t_j}\cdots x_{i_{t_j}}$ belongs to $\ass(I)$ .  
    \begin{example}
        Let $m=x_2x_3x_6x_8x_9x_{13}\in \K[x_1,\ldots,x_{15}]$. Then $T(m)=(2,3,5,6)$ and $IT(m)=(3,6,9,13)$. Hence the associated primes of $\sfBorel(m)$ are in one-to-one correspondence with the minimal monomial generators of $$\sfBorel(x_2x_3,x_3x_4x_5x_6,x_5x_6x_7x_8 x_9,x_6x_7x_8 x_9x_{10}x_{11}x_{12}x_{13}).$$
    \end{example}
    \subsection{The Waldschmidt constant.}
    We recall the definition of the Waldschmidt constant.

    Given a square-free monomial ideal $I\subseteq S$, let $I=P_1\cap P_2\cap\cdots \cap P_t$ be the minimal primary decomposition of $I$. The \emph{$m$-th symbolic power} of $I$ , denoted $I^{(m)}$, is the ideal $$I^{(m)}=P_1^m\cap P_2^m\cap \cdots\cap P_t^m.$$ For the more general definition of a symbolic power of an ideal, see \cite{IdealsofPowers-and-Powersofideals-}. 

    For any homogeneous ideal $L\subseteq S$, let $\alpha(L)$ denote the smallest degree of a generator of $L$. The \emph{Waldschmidt constant} of a square-free monomial ideal $L$ is defined as $$\widehat{\alpha}(L)=\lim _{m\rightarrow{\infty}}\frac{\alpha(L^{(m)})}{m}.$$ 
    The following result is a key tool in the study of the Waldschmidt constant of a square-free monomial ideal, which relates this invariant to a linear programming problem. 
    \begin{theorem}\cite[Theorem 3.2]{BCGHJNSTV2016-The-Waldschmidt-constant-for-sqfree-monomial-ideals}\label{thm: Waldschmidt constant is optimal solution of an LPP}
        Let $I\subseteq S$ be a square-free monomial ideal with the primary decomposition $I=P_1\cap P_2\cap\cdots \cap P_t$. Define the $t\times n$ matrix $A$ where $$A_{i,j}=\begin{cases}
            1 &\text{ if $x_j\in P_i$}\\0&\text{ if $x_j\notin P_i$. }
        \end{cases}$$
        Then $\widehat{{\alpha}}(I)$ is the optimal value of the linear program $$\min\{\mathbf{1}^T\mathbf y~|~A\mathbf y\geq \mathbf1,\mathbf{y\geq0}\}.$$
    \end{theorem}
     Note that rows of $A$ are indexed by associated primes of $I$ and columns by the variables in $S$. The matrix $A$ in Theorem \ref{thm: Waldschmidt constant is optimal solution of an LPP} is called the \emph{matrix of associated primes} of $I$.
     \begin{example}
         Let $m=x_{2}x_{n-1} x_{n}$, $n\geq 5$. Then $T(m)=(1,3)$ and $IT(m)=(2,n)$. The associated primes of $\sfBorel(m)$ are in one-to-one correspondence with the minimal monomial generators of $\sfBorel(x_1x_2,x_3\cdots x_{n})$. So the matrix $A$ of associated primes of $\sfBorel(m)$ is of size $(2n-2)\times n$.  
     \end{example}
Let $B$ be the submatrix of $A$ whose $j$-th row is corresponding to the associated prime $P_j=\langle x_{t_j},\ldots, x_{i_{t_j}}\rangle$ for each $0\leq j\leq k$. Authors in \cite[Lemma 3.3]{CKSV2022-On-the-Waldschmidt-constant-of-sqfree-principal-Borel-ideals} showed that we can bound the optimal solution of Theorem \ref{thm: Waldschmidt constant is optimal solution of an LPP} by considering only the submatrix $B$. Let $l$ be the smallest integer such that $i_{t_{l+1}}<s\leq i_{t_l}$, and $\mathbf x\in \R^n$. The condition of $\mathbf x_j$ being non-increasing up to the first $i_{t_l}$ terms in \cite[Lemma 3.3]{CKSV2022-On-the-Waldschmidt-constant-of-sqfree-principal-Borel-ideals} can be relaxed up to the degree $t_0=s$. The proofs follow identically to the original arguments and are thus omitted for brevity, and we have the following:
     \begin{lemma}\cite[Lemma 3.3]{CKSV2022-On-the-Waldschmidt-constant-of-sqfree-principal-Borel-ideals}\label{lemma: solving submatrix is enough}
         Let $m=x_{i_1}\cdots x_{i_s}$ be a square-free monomial with $T(m)=(t_k,\ldots,t_0)$, $IT(m)=(i_{t_k},\ldots,i_{t_0})$, and $I=\sfBorel(m)$. Let $B$ be the submatrix of $A$ as above. Suppose $\mathbf x\in \R^n$ is such that  
         \begin{enumerate}
         \item $B\mathbf{x}\geq \mathbf 1$,
             \item $\mathbf x_j\geq \mathbf x_{j+1}$ for $1\leq j\leq s-1$,
             \item $\mathbf x_s\geq \mathbf x_j$ for all $s+1\leq j\leq i_s$.
         \end{enumerate}
         Then $A\mathbf x\geq \mathbf 1$. 
     \end{lemma}
   
\section{Lower Bound}\label{section: lower bound}
Let $m=x_{i_1}\cdots x_{i_s}$ be a square-free monomial of degree $s$ in the polynomial ring $S=\K[x_1,\ldots,x_n]$ with $T(m)=(t_k,\ldots,t_0)$ and $IT(m)=(i_{t_k},\ldots,i_{t_0})$, and $I=\sfBorel(m)$. If $i_{t_{k}}\geq s$, then the Waldschmidt constant $\widehat{\alpha}(I)$ of $I$ is computed in \cite[Theorem 4.1]{CKSV2022-On-the-Waldschmidt-constant-of-sqfree-principal-Borel-ideals}. Henceforth, we assume that $i_{t_k} < s$ throughout this article, unless stated otherwise. Now, we construct a subsequence of $T(m)$ that is useful throughout the article.
  \begin{construction}\label{cons: definition of T'(m)}
       Let $m=x_{i_1}\cdots x_{i_s}$ be a square-free monomial of degree $s$ in $S$ with $T(m)=(t_k,\ldots,t_0)$ and $IT(m)=(i_{t_k},\ldots,i_{t_0})$.
  Now, we inductively construct a subsequence  $T'(m)=(t_1',t_2',\ldots,t_r')$ of $T(m)$ with $i_{t_j'}\leq s-1$ for each $1\leq j\leq r$ 
  as follows:
  
  Set $t_{1}'=t_k$. Now, suppose $t_{j}'$ is already constructed. Then $t_{j+1}'$ is the smallest in $T(m)$ such that $i_{t_{j}'}\leq s-1$ and $i_{{t_j'}}<t_{{j+1}}'$. Since we are following the notation in which the indices of $T(m)$ are decreasing, we use $T'(m)$ in the same order to avoid confusion. Therefore, we write $T'(m)=(t_{k_r},\ldots,t_{k_1})$, where $t_{k_j}=t_{r-j+1}'$ for each $j=r,r-1,\ldots,1$.

  \end{construction} 
  
\begin{remark}\label{remk: existence of l ...}
   With the notation as in Construction \ref{cons: definition of T'(m)}, observe that $i_{t_{k_1}}+1$ may not be a member of $T(m)$. Suppose $i_{t_{k_1}}+1\in T(m)$. Since $i_{t_{k_1}}+1\notin T'(m)$, it follows from the the construction of $T'(m)$ that $i_{i_{t_{k_1}}+1}\geq s$. Using above construction, note that $i_{t_{k_1}}\leq s-1$. If $i_{t_{k_1}}+1\notin T(m)$, then there exists $l$ such that $i_{t_{k_1}}+1<t_l$. Since $i_{t_{k_1}}<t_l$, by construction of $T'(m)$, we must have $i_{t_l}\geq s$. In conclusion, there exists $l$ such that $i_{t_{k_1}}+1\leq t_l$, and $i_{t_l}\geq s$. Also, if $l$ is largest such that $i_{t_{k_1}}+1\leq t_l$, then it follows from the construction of $T'(m)$ that $i_{t_l}\geq s$.

\end{remark}
  \begin{example}
      Let $m=x_{2}x_{4}x_{6}x_8x_{9}x_{10}x_{11}x_{12}x_{15}$. Then $T(m)=(1,2,3,8,9)$, $T'(m)=(1,3)$. In this case, $i_{t_{k_1}}+1\notin T(m)$ and $l=8$.
  \end{example}
The subsequence $T'(m)$ helps us to provide a lower bound for $\widehat{\alpha}(I)$, and we prove the following.
\begin{theorem}\label{thm: lower-bound in general}
    Let $m=x_{i_1}x_{i_2}\cdots x_{i_s}$ be a square-free monomial with $T'(m)$ as above, and $l$ be largest such that $i_{t_{k_1}}+1\leq t_l$, and suppose $I=\sfBorel(m)$. Then 
        $$\widehat{\alpha}(I)\geq 1+\frac{s-{{i_{t_{k_1}}-1}}}{i_{t_l}-t_l+1}+\frac{i_{t_k}}{i_{t_{k}}-t_k+1}+\sum_{j=1}^{r-1}\frac{i_{t_{k_j}}-i_{t_{k_{j+1}}}}{i_{t_{k_j}}-{t_{k_{j}}}+1}.$$  
\end{theorem}
\begin{proof}
By Theorem \ref{thm: description of primes}, let $P_j$ denote the associated prime of $I$ that is associated with the monomial $x_j\cdots x_{i_j}$ for $j=1,\ldots,s$. For any other prime ideal $P\in \ass(I)$, we write $P\sim P_j$ if the prime $P$ is associated to a monomial $u$ that is a Borel move of $x_j\cdots x_{i_j}$. Let $A$ be the matrix of associated primes of $I$. We denote $A_P$ for the row that is indexed by the associated prime $P$. For any $\mathbf x\in \R^{|\ass(I)|}$, we denote $\mathbf{x}_P$ for the coordinate of $\mathbf{x}$ corresponding to the associated prime $P$.  
     
     Define $\mathbf y\in \R^{\mid \ass(I)|}$ as follows:
     $$\mathbf y_P=
        \begin{cases}
\frac{1}{\binom{i_{t_{k_r}}-1}{t_{k_r}-1}}, & \text{ if $P\sim P_{t_{k_r}},$ } \\
           \dfrac{1}{\binom{i_{t_{k_j}}-i_{t_{k_{j+1}}}-1}{t_{k_j}-i_{t_{k_{j+1}}}-1}}, &\begin{array}{l}
               \text{if } P\sim P_{t_{k_j}} \text{ and } P\cap \{x_1,\ldots,x_{i_{t_{k_{j+1}}}}\}=\emptyset\\ \text{for each } j=r-1,\ldots,1,
           \end{array} \\
           \dfrac{1}{\binom{i_{t_l}-i_{t_{k_1}}-1}{i_{t_l}-t_l}}\frac{s-i_{t_{k_1}}-1}{i_{t_l}-i_{t_{k_1}}}, & \text{ if $P\sim P_{t_l}$ and $P\cap\{x_1,\ldots,x_{i_{t_{k_1}}}\}=\emptyset$},\\
           \dfrac{1}{\binom{i_{t_l}-i_{t_{k_1}}}{i_{t_l}-s+1}}, &
           \begin{array}{l}
           \text{if } P\sim P_{t_0} \text{, } P\cap \{x_1,\ldots,x_{i_{t_{k_{1}}}}\}=\emptyset \text{ and } \\\{x_{i_{t_l}+1},\ldots,x_{i_s}\}\subseteq P,
           \end{array}\\
          \quad \quad 0, & \text{ otherwise }. 
        \end{cases}$$
        Note that $i_{t_j}\geq t_j$ for $0\leq j\leq k$, and $i_{t_{k_1}}+1\leq t_l\leq s$. Therefore, all the entries of $\mathbf {y} $ are well-defined. By \cite[Lemma 3.1]{CKSV2022-On-the-Waldschmidt-constant-of-sqfree-principal-Borel-ideals}, we assume that $i_s=n$. Now, we show that the inequality $(A^T\mathbf y)_e\leq 1$ holds for each $1\leq e\leq i_s$. 
        By the definition of $A$,  we have $$(A^T\mathbf y)_e=\sum_{x_e\in P}\mathbf{y}_P.$$ 
        So, in order to compute $(A^T\mathbf y)_e$, we need to find primes $P$ satisfying $x_e\in P$ and $\mathbf{y}_P\neq 0$.

        Observe that any square-free monomial of degree $t$ in variables $\{x_1,\ldots,x_{i+t-1}\}$ is a Borel move of $u=x_{i}x_{i+1}\cdots x_{i+t-1}$. Thus, the number of Borel moves of $u$ is $\binom{i+t-1}{t}$. Further, the number of Borel moves of $u$ which are divisible by a fixed square-free monomial of degree $p$ is equal to $\binom{i+t-1-p}{t-p}$. 

        For $1\leq e\leq i_{t_k}$, it follows from the definition of $\mathbf y$ that $P\sim P_{t_k}$ are the only primes containing $x_e$ for which $\mathbf y_P\neq 0$. There are $\binom{i_{t_k}-1}{t_k-1}$ such primes. Therefore, we have 
    $$(A^T\mathbf y)_e=\binom{i_{t_k}-1}{t_k-1}\frac{1}{\binom{i_{t_k}-1}{t_k-1}}=1.$$
Fix $1\leq p\leq r-1$ and let $ i_{t_{k_{p+1}}}+1\leq  e\leq i_{t_{k_{p}}}  $. If $P\in\ass(I)$ containing $x_e$ for which $\y_P\neq 0$, then $P\sim P_{t_{k_{p}}}$ and $P\cap \{x_1,\ldots,x_{i_{t_{k_{p+1}}}}\}=\emptyset$. We have $\binom{i_{t_{k_{p}}}-i_{t_{k_{p+1}}}-1}{t_{k_p}-i_{t_{k_{p+1}}}-1}$ many such primes. So, we get 
    $$(A^T\mathbf y)_e=\binom{i_{t_{k_p}}-t_{k_{p+1}}-1}{t_{k_p}-i_{t_{k_{p+1}}}-1}\frac{1}{\binom{i_{t_{k_p}}-t_{k_{p+1}}-1}{t_{k_p}-i_{t_{k_{p+1}}}-1}}=1.$$
Now, suppose $i_{t_{k_1}}+1\leq e\leq i_{t_l}$. Then the only primes $P$ containing $x_e$ for which $\mathbf y_P\neq 0$ are $P\sim P_{t_l}$ or $P\sim P_{t_0}$ and $P\cap \{x_1,\ldots,x_{i_{t_{k_1}}}\}=\emptyset$. There are $\binom{i_{t_l}-i_{t_{k_1}}-1}{i_{t_l}-t_l}$ primes $P$ containing $x_e$ such that $P\sim P_{t_l}$ and $\mathbf y_P\neq 0$. Also, the number of  primes $P$ containing $x_e$ for which $\mathbf y_P\neq0$, and $P\sim P_{t_0}$ is $\binom{i_{t_l}-i_{t_{k_1}}-1}{i_{t_l}-s}$. Therefore, by the choice of $l$, and addition principle of counting, we get $$(A^T\mathbf y)_e=\frac{s-i_{t_{k_1}}-1}{i_{t_l}-i_{t_{k_1}}}+\binom{i_{t_l}-i_{t_{k_1}}-1}{i_{t_l}-s}\frac{1}{\binom{i_{t_l}-i_{t_{k_1}}}{i_{t_l}-s+1}}=1.$$

Finally, if $i_{t_{l}}+1\leq e\leq i_s$. Then $x_e$ appears in every prime $P$ for which $\mathbf y_P\neq 0$ and  $P\sim P_{t_0}$. There are $\binom{i_{t_l}-i_{t_{k_1}}}{i_{t_l}-s+1}$ such primes. Therefore, we get $$(A^T\mathbf y)_e=\binom{i_{t_l}-i_{t_{k_1}}}{i_{t_l}-s+1}\frac{1}{\binom{i_{t_l}-i_{t_{k_1}}}{i_{t_l}-s+1}}=1.$$
Thus, using the duality theorem, we have $\widehat{\alpha}(I)\geq \sum_{P\in \ass(I)}\mathbf y_P$. 
Now,
\begin{align*}
        \sum_P{\mathbf y_P}=&\binom{i_{t_k}}{t_k-1}\frac{1}{\binom{i_{t_k-1}}{t_k-1}}+\sum_{j=1}^{r-1}\binom{i_{t_{k_j}}-i_{t_{k_{j+1}}}}{t_{k_j}-i_{t_{k_{j+1}}}-1}\frac{1}{\binom{i_{t_{k_j}}-i_{t_{k_{j+1}}}-1}{t_{k_j}-i_{t_{k_{j+1}}}-1}}\\&+\binom{i_{t_l}-i_{t_{k_1}}}{i_{t_l}-t_l+1} \frac{1}{\binom{i_{t_l}-i_{t_{k_1}}-1}{i_{t_l}-t_l}}\frac{s-i_{t_{k_1}}-1}{i_{t_l}-i_{t_{k_1}}}+\binom{i_{t_l}-i_{t_{k_1}}}{i_{t_l}-s+1}\frac{1}{\binom{i_{t_l}-i_{t_{k_1}}}{i_{t_l}-s+1}}\\
        =&\frac{i_{t_k}}{i_{t_{k}}-t_k+1}+\sum_{j=1}^{r-1}\frac{i_{t_{k_j}}-i_{t_{k_{j+1}}}}{i_{t_{k_j}}-{t_{k_{j}}}+1}+\frac{s-i_{t_{k_1}}-1}{i_{t_l}-t_{l}+1}+1.
    \end{align*}
    Hence, the desired result holds.
\end{proof}
We illustrate Theorem \ref{thm: lower-bound in general} with an example. This also shows that the lower bound of Theorem \ref{thm: lower-bound in general} is attained for infinitely many Borel ideals, while the recursive lower bound given in \cite[Theorem 4.4]{CKSV2022-On-the-Waldschmidt-constant-of-sqfree-principal-Borel-ideals} is strict.   
\begin{example}\label{exm: Lower bound is attained for infinitely} 
Let $m=x_2x_5x_7x_{i_4}x_{i_5}x_{i_6}$ with $i_4\geq 9$ and $I=\sfBorel(m)$. Then 
    $T'(m)=(1)$ regardless of $i_4,i_5,$ and $i_6$ and $l=3$. So the vector $\mathbf y$ in Theorem \ref{thm: lower-bound in general} is 
    $$\mathbf y_P=\begin{cases}
        1, & \text{ if } P=(x_1,x_2),\\
        \frac{3}{5} , & \text{ if } P\sim(x_3,\ldots,x_7), \text{ and $\{x_1,x_2\}\cap P=\emptyset$},\\
        \frac{1}{\binom{5}{2}}, & \text{ if } P\sim P_{6} \text{ and } P\cap \{x_1,x_2\}=\emptyset,\\
           0, & \text{ otherwise }. 
    \end{cases}$$
    This gives $\widehat{\alpha}(I)\geq 2.6$. For the vector
    $$\mathbf y^T=\left(\frac{1}{2},\frac{1}{2},\frac{1}{5},\frac{1}{5},\frac{1}{5},\frac{1}{5},\frac{1}{5},\frac{1}{5},\frac{1}{5},\frac{1}{5},0,0,0,0,0,\ldots,0\right)\in \R^{i_6},$$ it can be easily verified that $A\mathbf y\geq 1$. Then, in view of Theorem \ref{thm: Waldschmidt constant is optimal solution of an LPP}, we get $\widehat{\alpha}(I)\leq 2.6$. Therefore, we have $\widehat{\alpha}(I)=2.6$, while the lower bound in \cite[Theorem 4.4]{CKSV2022-On-the-Waldschmidt-constant-of-sqfree-principal-Borel-ideals} is $2.5$
\end{example}   
\section{Upper Bound}\label{section: upper bound}
In this section, we provide an improved upper bound than \cite[Theorem 3.4]{CKSV2022-On-the-Waldschmidt-constant-of-sqfree-principal-Borel-ideals} for $\widehat{\alpha}(\sfBorel(m))$ under certain conditions.  Let $m=x_{i_1}x_{i_2}\cdots x_{i_s}$ be a square-free monomial with $i_1=1$. Then $\widehat{\alpha}(\sfBorel(m))=1+\widehat{\alpha}(\sfBorel(\frac{m}{x_1}))$. Thus, without loss of generality, we assume $i_1\geq 2$.

\begin{theorem}\label{thm: Upper bound on Waldschmidt constant with conditions}
       
   Let $m=x_{i_1}x_{i_2}\cdots x_{i_s}$ be a square-free monomial with $2\leq i_1\leq s-1$,
   % $T(m)=(t_0,t_1,\ldots,t_k)$ and $IT(m)=(x_{i_{t_0}},\ldots, x_{i_{t_k}})$
   and $I=\sfBorel(m)$. Suppose $i_j-i_1\geq (j-1)(i_2-i_1)$, for all $3\leq j\leq i_1+1$. Then $$\widehat{\alpha}(I)\leq 2+\frac{s-i_1-1}{i_1(i_2-i_1)}.$$ In addition, if $i_{i_1+1}\geq s-1$ and $i_1(i_2-i_1)=i_{i_1+1}-i_1$, then the equality holds.
    \end{theorem}
    \begin{proof}
    By \cite[Lemma 3.1]{CKSV2022-On-the-Waldschmidt-constant-of-sqfree-principal-Borel-ideals}, we can assume that $i_s=n$. 
        Define the vector $\mathbf y\in \mathbb R^{i_s}$ as follows:
        $$\mathbf y_i=\begin{cases}
        \frac{1}{i_1}, &\text{ for } i=1,\ldots,i_1,\\
        \frac{1}{i_1(i_2-i_1)}, &\text{ for } i_1+1\leq i\leq s-1+i_1(i_2-i_1),\\
        ~0, &\text{ otherwise. }
    \end{cases}$$
    Since $i_{i_1+1}-i_1\geq i_1(i_2-i_1)$ and $i_{i_1+1}\leq i_s-(s-i_1-1)$ , we have $s-1+i_1(i_2-i_1)\leq i_s$, and $\mathbf y\in \mathbb R^{i_s}$ is well-defined. Let $A$ be the matrix of associated primes of $I$. Denote $A_P$ the row of $A$ indexed by $P$. We now show that $A_P\mathbf y\geq 1$.
    
Denote $P_j=\langle x_j,\ldots,x_{i_j}\rangle$ for each $1\leq j\leq s$. 
Clearly, one has $A_{P_1}\mathbf y= 1$. 
Now, suppose $2\leq j\leq s$. If $i_j>s-1+i_1(i_2-i_1)$, then it follows from the fact $\frac{1}{i_1}\geq \frac{1}{i_1(i_2-i_1)}$ that $$A_{P_j}\mathbf y\geq \frac{s-1+i_1(i_2-i_1)-j+1}{i_1(i_2-i_1)}=1+\frac{s-j}{i_1(i_2-i_1)}\geq 1.$$
Now, assume that $i_j\leq s-1+i_1(i_2-i_1)$. Then, for $2\leq j\leq i_1$, we have $$A_{P_j}\mathbf y=\frac{i_1-j+1}{i_1}+\frac{i_j-i_1}{i_1(i_2-i_1)}\geq \frac{i_1-j+1}{i_1}+\frac{(j-1)(i_2-i_1)}{i_1(i_2-i_1)}=1.$$
For $i_1+1\leq j\leq s$, we have $A_{P_j}\mathbf y=\frac{i_j-j+1}{i_1(i_2-i_1)}\geq \frac{i_j-j+1}{i_{i_1+1}-i_1}\geq 1$, where the last inequality follows from the fact that $i_j\geq i_{i_1+1}+j-i_1-1$. Since $\frac{1}{i_1}\geq \frac{1}{i_1(i_2-i_1)}$, any prime $P$ associated to a Borel move of $x_j\cdots x_{i_j}$ for $1\leq j\leq s$ satisfies $A_P\mathbf y\geq 1 $. Thus, by Theorem \ref{thm: Waldschmidt constant is optimal solution of an LPP}, we get the desired result after adding all the entries of $\mathbf y$.

For the equality, we assume that $i_{i_1+1}\geq s-1$ and $i_1(i_2-i_1)=i_{i_1+1}-{i_1}$. For this, define $\y\in \R^{\mid \ass(I)|}$ as follows:
         $$\y_P=
        \begin{cases}
           \quad1, & \text{ if } P=P_1 
           ,\\
           \frac{s-i_1-1}{i_{i_1+1}-i_1}, & \text{ if } P=P_{i_1+1}
           ,\\
           \frac{1}{\binom{i_{i_1+1}-i_1}{s-i_1-1}}, & \text{ if } P\sim P_s, P\cap \{x_1,\ldots,x_{i_1}\}=\emptyset \text{ and }\{x_{i_{i_1+1}+1},\ldots,x_{i_s}\}\subset P,\\
           \quad0, & \text{ otherwise }. 
        \end{cases}$$ Since $i_{i_1+1}\geq s-1$, the entries of $\y$ are well-defined. We now show that for $1\leq e\leq i_s$, the inequality $(A^T\y)_e\leq 1$ is satisfied. 
        
        For each $1\leq e\leq i_1$, note that $P_1$ is the only prime containing $x_e$ for which $\y_P\neq 0$. So, we have $(A^T\y)_e=1$. Now, suppose $i_1+1\leq e\leq i_{i_1+1}$. Then primes $P$ containing $x_e$ for which $\y_P\neq 0$ are $P_{i_1+1}$ and $\binom{i_{i_1+1}-i_1-1}{s-i_1-1}$ many primes $P\sim (x_s,\ldots,x_{i_s})$. This gives $$(A^T\y)_e=\frac{s-i_1-1}{i_{i_1+1}-i_1}+\binom{i_{i_1+1}-i_1-1}{s-i_1-1}\frac{1}{\binom{i_{i_1+1}-i_1}{s-i_1-1}}=
        1.$$
        
     For $i_{i_1+1}+1\leq e\leq i_s$, the variable $x_e$ appears in every primes $P$ such that $P\sim (x_s,\ldots,x_{i_s})$ and $\y_P\neq 0$. There are $\binom{i_{i_1+1}-i_1}{s-i_1-1}$ many such primes. Therefore, we get $(A^T\y)_e=\frac{1}{\binom{i_{i_1+1}-i_1}{s-i_1-1}}\binom{i_{i_1+1}-i_1}{s-i_1-1}=1$. Thus, by Theorem \ref{thm: Waldschmidt constant is optimal solution of an LPP} and using duality, we obtain $$\widehat{\alpha}(I)\geq 1+\frac{s-i_1-1}{i_{i_1+1}-i_1}+\binom{i_{i_1+1}-i_1}{s-i_1-1}\frac{1}{\binom{i_{i_1+1}-i_1}{s-i_1-1}}=2+\frac{s-i_1-1}{i_{i_1+1}-i_1}.$$ This completes the proof.
    \end{proof}
    The following example provides an infinite family of square-free principal Borel ideals for which the bound given in the above theorem is attained, while the upper bound in \cite[Theorem 3.4]{CKSV2022-On-the-Waldschmidt-constant-of-sqfree-principal-Borel-ideals} is strict.
    \begin{example}
        Let $I=\sfBorel(x_2x_5x_{8}x_{i_4}\cdots x_{i_8})$ with $i_4\geq 10$. Then, it follows from Theorem \ref{thm: Upper bound on Waldschmidt constant with conditions} that $\hat{\alpha}(I)=\frac{17}{6}$, while the upper bound obtained in \cite[Theorem 3.4]{CKSV2022-On-the-Waldschmidt-constant-of-sqfree-principal-Borel-ideals} is $\frac{37}{12}$.
    \end{example}
\section{The Waldschmidt Constant}\label{section: the waldschmidt constant}
In this section, we show that the lower bound of Theorem \ref{thm: lower-bound in general} is attained under certain conditions, thereby determining the Waldschmidt constant. 
\begin{theorem}\label{thm: equality of waldschmidt constant}
    Let $m=x_{i_1}\cdots x_{i_s}$ with $T(m)$, $T'(m)$, and $l$ be the largest such that $i_{t_{k_1}}+1\leq t_l$, and let $I=\sfBorel(m)$. Suppose 
    
         \begin{equation}\label{eq: condn for equality formula}
         \frac{i_{t_{k_j}}-t_p+1}{i_{t_{k_j}}-t_{k_j}+1}+\frac{i_{t_p}-i_{t_{k_j}}}{i_{t_{k_{j-1}}}-t_{k_{j-1}}+1}\geq 1,
    \end{equation}
    for all $t_p\in T(m)$ satisfying ${t_{k_{j}}}<t_p<{t_{k_{j-1}}}$ for some $j=1,\ldots r$ \textnormal{(}$t_{k_{0}}=t_l$\textnormal{)}.
    Then $$
        \widehat{\alpha}(I)= 1+\frac{s-i_{t_{k_1}}-1}{i_{t_l}-t_l+1}+\frac{i_{t_k}}{i_{t_{k}}-t_k+1}+\sum_{j=1}^{r-1}\frac{i_{t_{k_j}}-i_{t_{k_{j+1}}}}{i_{t_{k_j}}-{t_{k_{j}}}+1}.$$
\end{theorem}
\begin{proof}
    By \cite[Lemma 3.1]{CKSV2022-On-the-Waldschmidt-constant-of-sqfree-principal-Borel-ideals}, we may assume that $i_s=n$. For each $1\leq e\leq r$, set  $a_e=i_{t_{k_e}}-t_{k_e}+1$ and $a=i_{t_l}-t_l+1$. By Theorem \ref{thm: lower-bound in general}, it suffices to show that $\widehat{\alpha}(I)$ is bounded above by the desired value. For this,
    consider $\mathbf y\in \R^{i_s}$, where 
     \begin{align*}
    \mathbf y^T&=\Biggl(
\underbrace{\frac{1}{a_r}, \ldots, \frac{1}{a_r}}_{i_{t_k}},
\underbrace{\frac{1}{a_{r-1}}, \ldots, \frac{1}{a_{r-1}}}_{i_{t_{k_{r-1}}}-i_{t_k}},\ldots,
\underbrace{\frac{1}{a_1}, \ldots, \frac{1}{a_1}}_{i_{t_{k_1}}-i_{t_{k_2}}}, \underbrace{\frac{1}{{a}},\ldots \frac{1}{a}}_{i_{t_l}-i_{t_{k_1}}}, \\&\qquad \underbrace{\frac{(t_{l-1}-t_{l})}{(i_{t_{l-1}}-i_{t_{l}})a},\ldots, \frac{(t_{l-1}-t_{l})}{(i_{t_{l-1}}-i_{t_{l}})a}}_{i_{t_{l-1}}-i_{t_{l}}},\ldots,\underbrace{\frac{(t_{0}-t_1)}{(i_{t_{0}}-i_{t_1})a},\ldots,\frac{(t_{0}-t_1)}{(i_{t_{0}}-i_{t_1})a}}_{i_{t_{0}}-i_{t_1}}\Bigg).
\end{align*}
 Let $A$ be the matrix of associated primes of $I$. Then, our aim is to show that $A\y\geq 1$. To establish this, we use Lemma \ref{lemma: solving submatrix is enough}. Since the sequence $i_{t_k}-t_k,\ldots,i_{t_0}-t_0$ is strictly increasing, we get that the first $i_{t_l}$ entries of $\mathbf y$ form a non-increasing sequence. Also, we have $$ \frac{t_{j}-t_{j+1}}{i_{t_{j}}-i_{t_{j+1}}}<1\text{ for each $j=0,\ldots,l-1$}.$$ By Remark \ref{remk: existence of l ...}, we get $i_{t_l}\geq s$, and hence $\frac{1}{a}\geq \mathbf y_j$ for all $j\geq s$. Therefore, the conditions $(2)$ and $(3)$ of Lemma \ref{lemma: solving submatrix is enough} holds for $\mathbf y$. 
 
 Let $B$ be the submatrix of $A$ where the $j$-th row $B_j$ of $B$ is corresponding to the associated prime $P_{t_j}=\langle x_{t_{j}},\ldots, x_{i_{t_{j}}}\rangle,$ i.e., 
 $$B_j=(\underbrace{0,\ldots,0}_{t_j-1},\underbrace{1,\ldots,1}_{i_{t_j}-t_j+1},0,\ldots,0).$$ We now show that $\mathbf y$ satisfies the condition $B\mathbf y\geq 1$.
 For each $1\leq e\leq r$, consider the $k_e$-th row of $B$. Then, we have
 $$B_{k_e}\mathbf y =\sum_{p=t_{k_e}}^{i_{t_{k_e}}}\mathbf y_p =\frac{i_{t_{k_e}}-t_{k_e}+1}{a_e}=1.$$ 

 Now, suppose there exists $t_p\in T(m)$ with $t_{k_j}<t_p<t_{k_{j-1}}$ for some $j$. We claim that $t_p\leq i_{t_{k_j}}$. On contrary, suppose $t_p> i_{t_{k_j}}$, then $i_{t_{k_j}}<t_p<t_{k_{j-1}}<i_{t_{k_{j-1}}}.$ Since ${t_{k_{j-1}}}\in T'(m)$, we get $i_{t_{k_{j-1}}}\leq s-1$. So, we have $i_{t_{k_j}}<t_p$ and $i_{t_{p}}< s-1$. Then by Construction \ref{cons: definition of T'(m)}, we must have $t_{k_{j-1}}\leq t_p<t_{k_{j-1}}$, which is absurd.
 Consequently, we get $i_{t_{k_{j+1}}}<t_{k_j}<t_p\leq i_{t_{k_j}}<i_{t_p}<i_{t_{k_{j-1}}}$. Then, we have 
$$B_p\mathbf y=\sum_{q=t_p}^{i_{t_p}}\mathbf y_q=\frac{i_{t_{k_j}}-t_p+1}{a_j}+\frac{i_{t_p}-i_{t_{k_j}}}{a_{j-1}}\geq 1,$$
 where the last inequality follows from Equation \eqref{eq: condn for equality formula}. Similarly, if $t_{k_1}<t_p<t_l$ is satisfied for some $t_p\in T(m)$, then $B_p\mathbf y\geq 1$. Hence, we have shown that $B_j\y\geq 1$ for all rows $B_j$ of $B$ with $j>l$. 

 Now, if $j\leq l$, then $t_j\geq t_l$. Recall that by Remark \ref{remk: existence of l ...}, we have $i_{t_l}\geq s$. Consequently, we get $t_l\leq t_j\leq s\leq i_{t_l}\leq i_{t_j}$. So, we get 
\begin{align*}
    B_j\mathbf y&=\sum_{p=t_j}^{i_{t_j}}\mathbf y_p=\frac{i_{t_l}-t_j+1}{a}+\frac{t_{l-1}-t_{l}}{a}+\cdots+\frac{t_j-t_{j+1}}{a}=1.
\end{align*}
Thus, $B_j\mathbf y\geq 1$ for all rows $B_j$ of $B$, and hence $A\y\geq 1$ by Lemma \ref{lemma: solving submatrix is enough}. As a consequence of Theorem~\ref{thm: Waldschmidt constant is optimal solution of an LPP}, we obtain $\widehat{\alpha}(I)\leq \sum_{p=1}^n\mathbf y_p$. Summing the entries of $\y$ yields the desired value.
\end{proof}
We provide an example to be more illuminating. 
\begin{example}
    Let $m=x_2x_5x_7x_{i_4}x_{i_5}x_{i_6}$ with $i_4\geq 9$. Then $$T(m)\in \{(1,2,3,6),(1,2,3,4,6),(1,2,3,5,6),(1,2,3,4,5,6)\},$$ $T'(m)=(1)$ and $t_l=3$. The vector $\mathbf{y}$ of Theorem \ref{thm: equality of waldschmidt constant} is
    \begin{align*}
    \mathbf y^T&=\Biggl(
\frac{1}{2},\frac{1}{2},\frac{1}{5},\frac{1}{5},\frac{1}{5},\frac{1}{5},\frac{1}{5},\underbrace{\frac{(t_{l-1}-t_{l})}{5(i_{t_{l-1}}-i_{t_{l}})},\ldots, \frac{(t_{l-1}-t_{l})}{5(i_{t_{l-1}}-i_{t_{l}})}}_{i_{t_{l-1}}-i_{t_{l}}}, \\ &\ldots,\underbrace{\frac{(t_{0}-t_1)}{5(i_{t_{0}}-i_{t_1})},\ldots,\frac{(t_{0}-t_1)}{5(i_{t_{0}}-i_{t_1})}}_{i_{t_{0}}-i_{t_1}}\Bigg).
\end{align*}
Observe that Equation \eqref{eq: condn for equality formula} is satisfied for $t_p=2\in T(m)$. Therefore, $\widehat{\alpha}(I)=2.6$. On the other hand, the upper bound given in \cite[Theorem 3.4]{CKSV2022-On-the-Waldschmidt-constant-of-sqfree-principal-Borel-ideals} is $2.75$.
\end{example}
\begin{corollary}\label{cor: Waldschmidt constant}
    Let $m=x_{i_1}\cdots x_{i_s}$ with $T(m)$, $T'(m)$, and $l$ be largest such that $i_{t_{k_1}+1}\leq t_l$, and let $I=\sfBorel(m)$. If $T'(m)$ is a truncation of $T(m)$ and there is no $t_p\in T(m)$ such that ${t_{k_1}}<t_p<t_l$, then  $$\widehat{\alpha}(I)= 1+\frac{s-i_{t_{k_1}}-1}{i_{t_l}-t_l+1}+\frac{i_{t_k}}{i_{t_{k}}-t_k+1}+\sum_{j=1}^{r-1}\frac{i_{t_{k_j}}-i_{t_{k_{j+1}}}}{i_{t_{k_j}}-{t_{k_{j}}}+1}.$$ 
\end{corollary}
\begin{proof}
    Immediately follows from Theorem \ref{thm: equality of waldschmidt constant}.
\end{proof}
\begin{example}
 Let $m=x_2x_4x_5x_7x_8x_9x_{11}x_{12}x_{13}x_{14}x_{16}x_{17}x_{18}x_{19}x_{20}$. Then, we have $T(m)=(1,3,6,10,15)$, $IT(m)=(2,5,9,14,20)$, and $T'(m)=(1,3,6,10)$. Observe that $T'(m)$ is a truncation of $T(m)$. Therefore, using Corollary \ref{cor: Waldschmidt constant}, we have $\widehat{\alpha}(\sfBorel(m))=5.$ 
\end{example}
To conclude this section, 
we define two sequences $a_j$ and $b_j$ for $j\geq 0$ as follows:$$a_j=a_{j-1}+j-1\text{ with }a_0=1\quad\text{ and }\quad b_j=b_{j-1}+j\text{ with }b_0=0.$$ It can be verified that $a_j$ and $b_j$ satisfy the following:
\begin{enumerate}
    \item $b_j-a_j=j-1$ for all $j\geq 0$, and 
    \item $a_j=b_{j-1}+1$ for all $j\geq 1$.
\end{enumerate}
Denote $\ell(T(m))$ the number of coordinates of $T(m)$. The next result is a generalization of \cite[Corollary 4.2]{CKSV2022-On-the-Waldschmidt-constant-of-sqfree-principal-Borel-ideals}. 
\begin{theorem}
    Let $k\in \mathbb N$. Then for every rational number $\frac{a}{b}\geq k$, there exists a square-free monomial $m\in S$ such that $\ell (T(m))=k$ and $\widehat{\alpha}(\sfBorel(m))=\frac{a}{b}.$
\end{theorem}
\begin{proof}
Write $\frac{a}{b}=(k-1)+\frac{c+b}{b}$ with $c\geq 0$. Let $m_1=x_b\cdots x_{c+b}$, and
\begin{align*}
    & m_j=x_{c+(j-1)b+a_{j-1}+1}\cdots x_{c+jb+b_{j-1}} \text{ for $2\leq j\leq k$}
\end{align*}
    Set $m=m_1m_2\cdots m_k$. Observe that $\deg(m_1)=c+1$, and for each $2\leq j\leq k$, we have $\deg(m_j)=b+{j-2}.$ Thus $\deg(m_1\cdots m_j)=c+(j-1)b+a_{j-1}$. By the construction of $m$, we have 
\begin{align*}
    &T(m)=\left(c+1,c+b+1,\ldots,c+(k-2)b+a_{k-2},c+(k-1)b+a_{k-1}\right), \\ 
    &IT(m)=(c+b,c+2b+b_1,\ldots,c+(k-1)b+b_{k-2},c+kb+b_{k-1}).
\end{align*}
Since $i_{c+(j-1)b+a_{j-1}}=c+jb+b_{j-1},$ and $b_{j-1}=a_j-1$ for $1\leq  j\leq k-1$, we get
$$T'(m)=\left(c+1,c+b+1,\ldots,c+(k-2)b+a_{k-2}\right).$$
Following the notation of Construction \ref{cons: definition of T'(m)}, note that $i_{t_{k_1}}=\deg(m)-1$.
   Therefore, by Corollary \ref{cor: Waldschmidt constant}, we get 
   \begin{align*}
       \widehat{\alpha}(\sfBorel(m))&=1+0+\frac{c+b}{b}+\sum_{j=2}^{k-1}\frac{c+jb+b_{j-1}-(c+(j-1)b+b_{j-2})}{(c+jb+a_{j})-(c+(j-1)b+a_{j-1})}\\
&=1+\frac{c+b}{b}+\sum_{j=2}^{k-1}\frac{b_{j-1}-b_{j-2}+b}{a_j-a_{j-1}+b}=
\frac{a}{b},
   \end{align*}
   where the last equality follows from the fact that $b_{j-1}-b_{j-2}=j-1=a_j-a_{j-1}$. Hence, the statement follows.
\end{proof}

\section{Square-free Borel and Lex-segment Ideals}\label{sec: square-free Borel and lex ideals}

In this section, we obtain lower and upper bounds for the Waldschmidt constant of square-free Borel ideals. For this, we introduce the following notation. 

  Let $m_1=x_{i_1}\cdots x_{i_s}$ and $m_2=x_{j_1}\cdots x_{j_s}$ be square-free monomials in the polynomial ring $S=\K[x_1,\ldots,x_n]$ over a field $\K$.  Denote $m_1\vee m_2:=x_{k_1}\cdots x_{k_s}$, where $k_r=\max\{i_r,j_r\}$ for $1\leq r\leq s$, and $m_1\wedge m_2:=x_{l_1}\cdots x_{ls}$ with $l_r=\min\{i_r,j_r\}$ for $1\leq r\leq s$. For $d\in \mathbb N$, the set of monomials of degree $d$ in $S$ is denoted by $\mon_d(S)$.

\begin{theorem}\label{thm: upper and lower bound of Borel ideals}
    Let $T=\{m_1,\ldots,m_t\}\subset \mon_s(S)$ be a subset consisting of square-free monomials, and $I=\sfBorel(T)$. Then $$\widehat{\alpha}(\sfBorel(m_1\vee \cdots\vee m_t))\leq \widehat{\alpha}(I)\leq \widehat{\alpha}(\sfBorel(m_1\wedge\cdots \wedge m_t)).$$
\end{theorem}
\begin{proof}
  Let $m_1\vee \cdots\vee m_t=x_{k_1}\cdots x_{k_s}$ and $m_1\wedge\cdots \wedge m_t=x_{l_1}\cdots x_{ls}$. Since $l_r$ is the minimum among the $r$-th indices of elements of $T$, it follows that $m_1\wedge\cdots \wedge m_t$ can be obtained from a Borel move of every element from $T$. Similarly, every element of $T$ can be obtained from a Borel move of $m_1\vee \cdots\vee m_t$. Consequently, we get $$\sfBorel(m_1\wedge\cdots \wedge m_t)\subset I\subset \sfBorel(m_1\vee \cdots\vee m_t). $$
  From the fact that if $J \subset K$, then $\widehat{\alpha}(J)\geq \widehat{\alpha}(K),$ the desired result follows. 
\end{proof}
Now, we proceed to study the Waldschmidt constant of square-free lex-segment ideals.
\begin{definition}\label{def: square-free lex ideal}
 Let $S=\K[x_1,\ldots,x_n]$ be a polynomial ring over a field $\K$. Consider the monomial order $<_{\lex}$ on $S$ induced by the ordering $x_1>x_2>\cdots>x_n>1$. For a square-free monomial $m\in \mon_s(S)$, the \emph{square-free lex-segment ideal} of $m$ is defined as $$\sflex(m)=\langle u\in\mon_s(S):\sqrt{u}=u, ~m\leq_{\lex}u\rangle.$$
 \end{definition}
 \begin{theorem}\label{thm: the Waldschmidt constant of lex ideals}
  
     Let $m=x_{i_1}\cdots x_{i_s}$ be a square-free monomial in $S$. Then 
     \begin{enumerate}
         \item $\widehat{\alpha}(\sflex(m))\leq 1+\frac{s-1}{i_1}$, and 
         \item $\widehat{\alpha}(\sflex(m))\geq 
         \begin{cases}
             1+\frac{s-1}{i_1}, \qquad\text{ if } i_{1}\geq s\\
             2+\frac{s-i_1-1}{n-s+1},\quad\text{ if } i_{1}\leq s-1.
         \end{cases}$
     \end{enumerate}
     In particular, if $i_{1}\geq s-1$, then $\widehat{\alpha}(\sflex(m))=1+\frac{s-1}{i_1}.$
    
 \end{theorem}
 \begin{proof}
In order to prove the theorem, we first claim that $$\sflex(m)\subset \sfBorel(x_{i_1}x_{n-s+2}x_{n-s+3}\cdots x_n) .$$ Let $u=x_{j_1}x_{j_2}\cdots x_{j_s}\in \sflex(m)$. Then $m\leq_{\lex}u$. Thus there exists $p$ such that $i_p> j_p$ for some $1\leq p\leq s$, and $i_l=j_l$ for $1\leq l<p$. Consequently, the monomial $u$ can be obtained from a Borel move of $x_{i_1}x_{n-s+2}x_{n-s+3}\cdots x_n$. Thus, we have \begin{equation}\label{eq: lower bound of lex}
        \widehat{\alpha}(\sfBorel(x_{i_1}x_{n-s+2}x_{n-s+3}\cdots x_n))\leq \widehat{\alpha}(\sflex(m)). 
    \end{equation}
    Since $\sfBorel(x_{i_1}x_{i_1+1}\cdots x_{i_1+s-1})\subset \sflex(m)$, it follows that 
    \begin{equation}\label{eq: upper bound of lex}
        \widehat{\alpha}(\sflex(m))\leq \widehat{\alpha}(\sfBorel(x_{i_1}x_{i_1+1}\cdots x_{i_1+s-1})).
    \end{equation}
    
    $(1)$ Immediately follows from Equation \eqref{eq: upper bound of lex} and \cite[Corollary 4.2]{CKSV2022-On-the-Waldschmidt-constant-of-sqfree-principal-Borel-ideals}.

    $(2)$ If $i_{1}\geq s$, then the result follows from Equation \eqref{eq: lower bound of lex} and \cite[Theorem 4.1]{CKSV2022-On-the-Waldschmidt-constant-of-sqfree-principal-Borel-ideals}. Otherwise, the result follows from Equation \eqref{eq: lower bound of lex} and Theorem \ref{thm: lower-bound in general}.
 \end{proof}
 \section*{Acknowledgments}
 The authors acknowledge the use of the computer algebra system Macaulay2 \cite{M2} and the online platform SageMath \cite{sagemath} for testing their computations. 
\subsection*{Conflict of interest} The authors declare that they have no known competing financial interests or personal relationships that could have appeared to influence the work reported in this paper.
\nocite{Monomial-ideals-HH}
\bibliographystyle{plain}
	\bibliography{bib}

@article {Borel-Generators-FMJ-2011,
    AUTHOR = {Francisco, Christopher A. and Mermin, Jeffrey and Schweig,
              Jay},
     TITLE = {Borel generators},
   JOURNAL = {J. Algebra},
  FJOURNAL = {Journal of Algebra},
    VOLUME = {332},
      YEAR = {2011},
     PAGES = {522--542},
      ISSN = {0021-8693,1090-266X},
   MRCLASS = {13F20 (05E40 13D02 13D40)},
  MRNUMBER = {2774702},
MRREVIEWER = {Benjamin\ P.\ Richert},
       DOI = {10.1016/j.jalgebra.2010.09.042},
       URL = {https://doi.org/10.1016/j.jalgebra.2010.09.042},
}

@article {BCGHJNSTV2016-The-Waldschmidt-constant-for-sqfree-monomial-ideals,
    AUTHOR = {Bocci, Cristiano and Cooper, Susan and Guardo, Elena and
              Harbourne, Brian and Janssen, Mike and Nagel, Uwe and
              Seceleanu, Alexandra and Van Tuyl, Adam and Vu, Thanh},
     TITLE = {The {W}aldschmidt constant for squarefree monomial ideals},
   JOURNAL = {J. Algebraic Combin.},
  FJOURNAL = {Journal of Algebraic Combinatorics. An International Journal},
    VOLUME = {44},
      YEAR = {2016},
    NUMBER = {4},
     PAGES = {875--904},
      ISSN = {0925-9899,1572-9192},
   MRCLASS = {13F20 (13A02 13F55 14N05 14N20)},
  MRNUMBER = {3566223},
MRREVIEWER = {Christopher\ A.\ Francisco},
       DOI = {10.1007/s10801-016-0693-7},
       URL = {https://doi.org/10.1007/s10801-016-0693-7},
}

@article {CKSV2022-On-the-Waldschmidt-constant-of-sqfree-principal-Borel-ideals,
    AUTHOR = {Camps Moreno, Eduardo and Kohne, Craig and Sarmiento, Eliseo
              and Van Tuyl, Adam},
     TITLE = {On the {W}aldschmidt constant of square-free principal {B}orel
              ideals},
   JOURNAL = {Proc. Amer. Math. Soc.},
  FJOURNAL = {Proceedings of the American Mathematical Society},
    VOLUME = {150},
      YEAR = {2022},
    NUMBER = {10},
     PAGES = {4145--4157},
      ISSN = {0002-9939,1088-6826},
   MRCLASS = {13F20 (13F55)},
  MRNUMBER = {4470164},
MRREVIEWER = {Aryampilly\ V.\ Jayanthan},
       DOI = {10.1090/proc/16082},
       URL = {https://doi.org/10.1090/proc/16082},
}

@book {IdealsofPowers-and-Powersofideals-,
    AUTHOR = {Carlini, Enrico and H\`a, Huy T\`ai and Harbourne, Brian and
              Van Tuyl, Adam},
     TITLE = {Ideals of powers and powers of ideals},
    SERIES = {Lecture Notes of the Unione Matematica Italiana},
    VOLUME = {27},
      NOTE = {Intersecting algebra, geometry, and combinatorics,
              With a foreword by Alfio Ragusa},
 PUBLISHER = {Springer, Cham},
      YEAR = {[2020] \copyright 2020},
     PAGES = {xix+159},
      ISBN = {978-3-030-45246-9; 978-3-030-45247-6},
   MRCLASS = {13-02 (11P05 13F20 14-02 14N07)},
  MRNUMBER = {4233193},
MRREVIEWER = {Jorge\ Neves},
       DOI = {10.1007/978-3-030-45247-6},
       URL = {https://doi.org/10.1007/978-3-030-45247-6},
}

@incollection {Waldschmidt-introduced,
    AUTHOR = {Waldschmidt, Michel},
     TITLE = {Propri\'et\'es arithm\'etiques de fonctions de plusieurs
              variables. {II}},
 BOOKTITLE = {S\'eminaire {P}ierre {L}elong ({A}nalyse) (ann\'ee 1975/76);
              {J}ourn\'ees sur les {F}onctions {A}nalytiques ({T}oulouse,
              1976)},
    SERIES = {Lecture Notes in Math.},
    VOLUME = {Vol. 578},
     PAGES = {108--135},
 PUBLISHER = {Springer, Berlin-New York},
      YEAR = {1977},
   MRCLASS = {10F35},
  MRNUMBER = {453659},
MRREVIEWER = {P.\ Bundschuh},
}

@Misc{M2,
          author = {Grayson, Daniel R. and Stillman, Michael E.},
          title = {Macaulay2, a software system for research in algebraic geometry},
          howpublished = {Available at \url{http://www2.macaulay2.com}}
        }

@article {Comparing-powers-and-symbolic-powers-of-ideals-B_H,
    AUTHOR = {Bocci, Cristiano and Harbourne, Brian},
     TITLE = {Comparing powers and symbolic powers of ideals},
   JOURNAL = {J. Algebraic Geom.},
  FJOURNAL = {Journal of Algebraic Geometry},
    VOLUME = {19},
      YEAR = {2010},
    NUMBER = {3},
     PAGES = {399--417},
      ISSN = {1056-3911,1534-7486},
   MRCLASS = {13F20 (13A15)},
  MRNUMBER = {2629595},
MRREVIEWER = {Irena\ Swanson},
       DOI = {10.1090/S1056-3911-09-00530-X},
       URL = {https://doi.org/10.1090/S1056-3911-09-00530-X},
}

@article {Steiner-systems-and-configrations-of-points-BFGM,
    AUTHOR = {Ballico, Edoardo and Favacchio, Giuseppe and Guardo, Elena and
              Milazzo, Lorenzo},
     TITLE = {Steiner systems and configurations of points},
   JOURNAL = {Des. Codes Cryptogr.},
  FJOURNAL = {Designs, Codes and Cryptography. An International Journal},
    VOLUME = {89},
      YEAR = {2021},
    NUMBER = {2},
     PAGES = {199--219},
      ISSN = {0925-1022,1573-7586},
   MRCLASS = {51E10 (13F20 13F55 14G50 94B27)},
  MRNUMBER = {4212956},
MRREVIEWER = {Svetlana\ Todorova\ Topalova},
       DOI = {10.1007/s10623-020-00815-x},
       URL = {https://doi.org/10.1007/s10623-020-00815-x},
}

@article {Chudnovsky-conjecture-and-the-stable-Harbourne-Huneke-containment-BGHN,
    AUTHOR = {Bisui, Sankhaneel and Grifo, Elo\'isa and H\`a, Huy T\`ai and
              Nguyen, Thai Thanh},
     TITLE = {Chudnovsky's conjecture and the stable {H}arbourne-{H}uneke
              containment},
   JOURNAL = {Trans. Amer. Math. Soc. Ser. B},
  FJOURNAL = {Transactions of the American Mathematical Society. Series B},
    VOLUME = {9},
      YEAR = {2022},
     PAGES = {371--394},
      ISSN = {2330-0000},
   MRCLASS = {14N20 (13F20 14C20)},
  MRNUMBER = {4427103},
MRREVIEWER = {Fei\ Ye},
       DOI = {10.1090/btran/103},
       URL = {https://doi.org/10.1090/btran/103},
}

@article {Powers-of-principal-Q-Borel-ideals-CKCSV,
    AUTHOR = {Camps-Moreno, Eduardo and Kohne, Craig and Sarmiento, Eliseo
              and Van Tuyl, Adam},
     TITLE = {Powers of principal {$Q$}-{B}orel ideals},
   JOURNAL = {Canad. Math. Bull.},
  FJOURNAL = {Canadian Mathematical Bulletin. Bulletin Canadien de
              Math\'ematiques},
    VOLUME = {65},
      YEAR = {2022},
    NUMBER = {3},
     PAGES = {633--652},
      ISSN = {0008-4395,1496-4287},
   MRCLASS = {13F55 (05E40)},
  MRNUMBER = {4472492},
MRREVIEWER = {Satoshi\ Murai},
       DOI = {10.4153/S0008439521000606},
       URL = {https://doi.org/10.4153/S0008439521000606},
}

@book {Monomial-ideals-HH,
    AUTHOR = {Herzog, J\"urgen and Hibi, Takayuki},
     TITLE = {Monomial ideals},
    SERIES = {Graduate Texts in Mathematics},
    VOLUME = {260},
 PUBLISHER = {Springer-Verlag London, Ltd., London},
      YEAR = {2011},
     PAGES = {xvi+305},
      ISBN = {978-0-85729-105-9},
   MRCLASS = {13D02 (05E40 13D40 13F55 13P10)},
  MRNUMBER = {2724673},
MRREVIEWER = {Rahim\ Zaare-Nahandi},
       DOI = {10.1007/978-0-85729-106-6},
       URL = {https://doi.org/10.1007/978-0-85729-106-6},
}

@article {The-Waldschmidt-Constant-of-Special-K-configrations-in-$P^n$-CGS,
    AUTHOR = {Catalisano, Maria Virginia and Guardo, Elena and Shin,
              Yong-Su},
     TITLE = {The {W}aldschmidt constant of special {$\Bbbk$}-configurations
              in {$\mathbb P^n$}},
   JOURNAL = {J. Pure Appl. Algebra},
  FJOURNAL = {Journal of Pure and Applied Algebra},
    VOLUME = {224},
      YEAR = {2020},
    NUMBER = {10},
     PAGES = {106341, 28},
      ISSN = {0022-4049,1873-1376},
   MRCLASS = {14M05 (13F20)},
  MRNUMBER = {4093063},
MRREVIEWER = {Jorge\ Neves},
       DOI = {10.1016/j.jpaa.2020.106341},
       URL = {https://doi.org/10.1016/j.jpaa.2020.106341},
}

@incollection {Singular-points-on-complex-hypersurfaces-and-multidimensional-Schwartz-lemma-Chudnovsky,
    AUTHOR = {Chudnovsky, G. V.},
     TITLE = {Singular points on complex hypersurfaces and multidimensional
              {S}chwarz lemma},
 BOOKTITLE = {Seminar on {N}umber {T}heory, {P}aris 1979--80},
    SERIES = {Progr. Math.},
    VOLUME = {12},
     PAGES = {29--69},
 PUBLISHER = {Birkh\"auser, Boston, MA},
      YEAR = {1981},
      ISBN = {3-7643-3035-X},
   MRCLASS = {32A22 (10F35 14J17)},
  MRNUMBER = {633888},
MRREVIEWER = {Michel\ Waldschmidt},
}

@book {Numbers-transcendants-..-Waldschmidt,
    AUTHOR = {Waldschmidt, Michel},
     TITLE = {Nombres transcendants et groupes alg\'ebriques},
    SERIES = {Ast\'erisque},
    VOLUME = {69-70},
      NOTE = {With appendices by Daniel Bertrand and Jean-Pierre Serre,
              With an English summary},
 PUBLISHER = {Soci\'et\'e{} Math\'ematique de France, Paris},
      YEAR = {1979},
     PAGES = {218},
   MRCLASS = {10F35 (14G25 14L15)},
  MRNUMBER = {570648},
MRREVIEWER = {Pierre\ Lelong},
}

@incollection {Estimation-L2-pour-..-Skoda,
    AUTHOR = {Skoda, H.},
     TITLE = {Estimations {$L\sp{2}$} pour l'op\'erateur {$\overline
              \partial $} et applications arithm\'etiques},
 BOOKTITLE = {S\'eminaire {P}ierre {L}elong ({A}nalyse) (ann\'ee 1975/76);
              {J}ourn\'ees sur les {F}onctions {A}nalytiques ({T}oulouse,
              1976)},
    SERIES = {Lecture Notes in Math.},
    VOLUME = {Vol. 578},
     PAGES = {314--323},
 PUBLISHER = {Springer, Berlin-New York},
      YEAR = {1977},
   MRCLASS = {32F05 (10F35)},
  MRNUMBER = {460723},
MRREVIEWER = {Michel\ Waldschmidt},
}

@article {The-resurgence-of-ideals-of-points-and-the-containment-problem-Bocci-Harbourne,
    AUTHOR = {Bocci, Cristiano and Harbourne, Brian},
     TITLE = {The resurgence of ideals of points and the containment
              problem},
   JOURNAL = {Proc. Amer. Math. Soc.},
  FJOURNAL = {Proceedings of the American Mathematical Society},
    VOLUME = {138},
      YEAR = {2010},
    NUMBER = {4},
     PAGES = {1175--1190},
      ISSN = {0002-9939,1088-6826},
   MRCLASS = {14C20 (14N05)},
  MRNUMBER = {2578512},
       DOI = {10.1090/S0002-9939-09-10108-9},
       URL = {https://doi.org/10.1090/S0002-9939-09-10108-9},
}

@article {A-canonical-linear-system-associated-to-adjoint-divisors-in-characteristic-p>0-Schwede,
    AUTHOR = {Schwede, Karl},
     TITLE = {A canonical linear system associated to adjoint divisors in
              characteristic {$p>0$}},
   JOURNAL = {J. Reine Angew. Math.},
  FJOURNAL = {Journal f\"ur die Reine und Angewandte Mathematik. [Crelle's
              Journal]},
    VOLUME = {696},
      YEAR = {2014},
     PAGES = {69--87},
      ISSN = {0075-4102,1435-5345},
   MRCLASS = {14C20 (13A35 14F18 14G17)},
  MRNUMBER = {3276163},
MRREVIEWER = {Yukihide\ Takayama},
       DOI = {10.1515/crelle-2012-0087},
       URL = {https://doi.org/10.1515/crelle-2012-0087},
}

@article {Are-symbolic-powers-highly-evolved-Harbourne-Huneke,
    AUTHOR = {Harbourne, Brian and Huneke, Craig},
     TITLE = {Are symbolic powers highly evolved?},
   JOURNAL = {J. Ramanujan Math. Soc.},
  FJOURNAL = {Journal of the Ramanujan Mathematical Society},
    VOLUME = {28A},
      YEAR = {2013},
     PAGES = {247--266},
      ISSN = {0970-1249,2320-3110},
   MRCLASS = {13F20 (13A02 13C05 13C10 14C20 14N05)},
  MRNUMBER = {3115195},
MRREVIEWER = {N.\ Mohan Kumar},
}

@article {Symbolic-powers-of-monomial-ideals-Cooper-Embree-Robert-Ha-Hoefel,
    AUTHOR = {Cooper, Susan M. and Embree, Robert J. D. and H\`a, Huy T\`ai
              and Hoefel, Andrew H.},
     TITLE = {Symbolic powers of monomial ideals},
   JOURNAL = {Proc. Edinb. Math. Soc. (2)},
  FJOURNAL = {Proceedings of the Edinburgh Mathematical Society. Series II},
    VOLUME = {60},
      YEAR = {2017},
    NUMBER = {1},
     PAGES = {39--55},
      ISSN = {0013-0915,1464-3839},
   MRCLASS = {13F20 (13A02 14N05)},
  MRNUMBER = {3589840},
MRREVIEWER = {Mike\ Janssen},
       DOI = {10.1017/S0013091516000110},
       URL = {https://doi.org/10.1017/S0013091516000110},
}

@article {Waldschmidt-constants-for-Stanley-Reisner-ideals-of-a-class-of-simplicial-complexes,
    AUTHOR = {Bocci, Cristiano and Franci, Barbara},
     TITLE = {Waldschmidt constants for {S}tanley-{R}eisner ideals of a
              class of simplicial complexes},
   JOURNAL = {J. Algebra Appl.},
  FJOURNAL = {Journal of Algebra and its Applications},
    VOLUME = {15},
      YEAR = {2016},
    NUMBER = {7},
     PAGES = {1650137, 13},
      ISSN = {0219-4988,1793-6829},
   MRCLASS = {13F55 (13F20 52B10)},
  MRNUMBER = {3528565},
MRREVIEWER = {Christopher\ A.\ Francisco},
       DOI = {10.1142/S0219498816501371},
       URL = {https://doi.org/10.1142/S0219498816501371},
}

@manual{sagemath,
  label        = {Sag95},
  author       = {{The Sage Developers}},
  title        = {{S}ageMath, the {S}age {M}athematics {S}oftware {S}ystem, {V}ersion 10.3},
  url          = {https://www.sagemath.org},
  year         = {2024},
  note         = {DOI 10.5281/zenodo.10841614},
}
\end{document}